\documentclass{article}
\usepackage[margin=1in]{geometry}

\usepackage{amsmath}
\usepackage{amsfonts}
\usepackage{amssymb}
\usepackage{amsthm}
\usepackage{graphics}

\usepackage{marvosym}
\newcommand{\envelope}{\raisebox{-.5pt}{\scalebox{1.45}{\Letter}}\kern-1.7pt}

\usepackage[margin=1in]{geometry}

\newtheorem{thm}{Theorem}[section]
\newtheorem{prop}[thm]{Proposition}

\newtheorem{lem}[thm]{Lemma}

\newtheorem{defn}[thm]{Definition}

\newtheorem{conj}{Conjecture}

\newtheorem{perty}{Property}

\newcommand{\A}{\mathcal{A}}
\newcommand{\N}{\mathbb{N}}
\newcommand{\Z}{\mathbb{Z}}
\newcommand{\Q}{\mathbb{Q}}
\newcommand{\G}{\mathcal{G}}

\newcommand{\ol}[1]{\overline{#1}}

\newcommand{\la}{\langle}
\newcommand{\ra}{\rangle}

\DeclareMathOperator*{\bigdoublewedge}{\bigwedge\mkern-15mu\bigwedge}
\DeclareMathOperator*{\bigdoublevee}{\bigvee\mkern-15mu\bigvee}

\title{Scott sentences for certain groups}
\author{Julia F.\ Knight \and Vikram Saraph}

\begin{document}
\maketitle

\begin{abstract}
We give Scott sentences for certain computable groups, and we use index set calculations as a way of checking that our Scott sentences are as simple as possible. We consider finitely generated groups and torsion-free abelian groups of finite rank.  For both kinds of groups, the computable ones all have computable $\Sigma_3$ Scott sentences.  Sometimes we can do better.  In fact, the computable finitely generated groups that we have studied all have Scott sentences that are ``computable $d$-$\Sigma_2$'' (the conjunction of a computable $\Sigma_2$ sentence and a computable $\Pi_2$ sentence).  In \cite{freegroups}, this was shown for the finitely generated free groups. Here we show it for all finitely generated abelian groups, and for the infinite dihedral group.\footnote{Since this paper was written (in 2013), there have been further results obtained.  Ho \cite{Ho} showed that the computable finitely generated groups in several further classes also have computable $d$-$\Sigma_2$ Scott sentences.  He also showed that for a finitely generated computable group $G$, if there is a computable $\Sigma_2$ formula that defines a non-empty set of generating tuples, then there is a computable $d$-$\Sigma_2$ Scott sentence.  
Still more recently, Harrison-Trainor and Ho \cite{HH} characterized the finitely generated groups that have a $d$-$\Sigma_2$ Scott sentence as those that do not have a generating tuple $\bar{a}$ and a further tuple $\bar{b}$, not a generating tuple, such that all existential formulas true of $\bar{b}$ are true of $\bar{a}$.  They gave an example of a computable finitely generated group that does not have a $d$-$\Sigma_2$ Scott sentence.  Alvir, Knight, and McCoy \cite{AKM} gave a different characterization.  For a finitely generated group $G$, there is a $d$-$\Sigma_2$ Scott sentence if and only if for every (some) generating tuple $\bar{a}$, the orbit of $\bar{a}$ is defined by a $\Pi_1$ formula.  For a computable finitely generated group, there is a computable $d$-$\Sigma_2$ Scott sentence if and only if for every (some) generating tuple $\bar{a}$, the orbit of $\bar{a}$ is defined by a computable $\Pi_1$ formula.}  Among the computable torsion-free abelian groups of finite rank, we focus on those of rank $1$.  These are exactly the additive subgroups of $\mathbb{Q}$.  We show that for some of these groups, the computable $\Sigma_3$ Scott sentence is best possible, while for others, there is a computable $d$-$\Sigma_2$ Scott sentence.\end{abstract}

\section{Introduction}

We can describe any finite group, or other finite structure for a finite language, up to isomorphism, by a single elementary (finitary) first order sentence.  However, for an infinite structure $\mathcal{A}$, the complete elementary first order theory of $\mathcal{A}$ always has models not isomorphic to $\mathcal{A}$.  Some countable groups, such as the infinite group in which every nonidentity element has order $p$, have an $\aleph_0$-categorical elementary first order theory.  However, for most countable groups, the theory has non-isomorphic countable models.  Sela \cite{S} showed that all non-abelian free groups have the same elementary first order theory (see also work of Kharlampovich and Myasnikov \cite{KM}).  

In this paper, we use infinitary formulas to describe our groups.  Recall that for language $L$, the logic $L_{\omega_1\omega}$ allows countably infinite conjunctions and disjunctions.  Scott \cite{scott} showed that for any countable structure $\mathcal{A}$ for a countable language $L$, there is a sentence of $L_{\omega_1\omega}$ whose countable models are exactly the isomorphic copies of $\mathcal{A}$.  Such a sentence is called a \emph{Scott sentence} for $\mathcal{A}$.  
The \emph{computable infinitary formulas} are formulas of $L_{\omega_1\omega}$ in which the infinite disjunctions and conjunctions are over computably enumerable (c.e.) sets.  We consider computable infinitary formulas to be comprehensible, even though they may be infinitely long.  

We consider only infinitary formulas in ``normal form'', with negations brought inside, next to the atomic formulas.  Computable infinitary formulas in normal form are classified as ``computable $\Sigma_\alpha$'' or ``computable $\Pi_\alpha$'', according to the number of alternations of 
$\bigdoublevee/(\exists)$ with $\bigdoublewedge/(\forall)$.     
\begin{enumerate}
\item $\varphi(\ol{x})$ is \emph{computable $\Sigma_0$} and \emph{computable $\Pi_0$} if it is finitary and quantifier-free;
\item for a computable ordinal $\alpha > 0$, 
\begin{enumerate}
\item  $\varphi(\ol{x})$ is \emph{computable $\Sigma_\alpha$} if it is a c.e.\  disjunction of formulas of the form $(\exists \ol{y}) \psi(\ol{x}, \ol{y})$, where each $\psi$ is computable $\Pi_\beta$ for some $\beta<\alpha$;
\item $\varphi(\ol{x})$ is $\Pi_\alpha$ if it is a c.e.\ conjunction of formulas of the form $(\forall \ol{y}) \psi(\ol{x}, \ol{y})$, where each $\psi$ is computable $\Sigma_\beta$ for some $\beta<\alpha$.
\end{enumerate}
\end{enumerate}

In computable structure theory, it is standard to identify a structure with its atomic diagram.  Thus, a structure $\A$ is \emph{computable} if the atomic diagram $D(\A)$ is computable.  This means that the universe is computable and the functions and relations are uniformly computable.  The usual language of groups has a binary operation symbol for the group operation, a unary operation symbol for the inverse, and a constant for the identity.  We note that if a group $\mathcal{G}$ has computable universe and computable group operation, then the inverse operation is also computable, so $\mathcal{G}$ is computable.  For a finitely generated group 
$\mathcal{G}$, being computable is the same as having solvable word problem \cite{WordProb}.  Henceforth, the groups that we consider are all computable unless we say otherwise.  

For a computable structure $\mathcal{A}$, the complexity of an ``optimal'' (i.e., simplest possible) Scott sentence provides an intrinsic measure of internal complexity of $\mathcal{A}$.  Not every computable structure has a computable infinitary Scott sentence.  One example is the \emph{Harrison $p$-group}, where this is a computable abelian $p$-group of length $\omega_1^{CK}$ with all infinite Ulm invariants, and with a divisible part of infinite dimension \cite{GHKS}.  Many familiar kinds of computable structures \emph{do} have computable infinitary Scott sentences.  This is true of the groups that we consider.  To show that the Scott sentences we find are optimal, we calculate the complexity of the ``index set''.     

\begin{defn} [Index set]\

\begin{enumerate}

\item  For a computable structure $\mathcal{A}$, the \emph{index set}, denoted by $I(\mathcal{A})$, is the set of all indices $e \in \N$ such that 
$\varphi_e = \chi_{D(\mathcal{B})}$ for some $\mathcal{B}\cong\mathcal{A}$.  

\item  For a class $K$ of structures (for a fixed language), closed under isomorphism, the \emph{index set}, denoted by $I(K)$, is the set of all $e$ such that 
$\varphi_e = \chi_{D(\mathcal{B})}$ for some $\mathcal{B}\in K$. 

\end{enumerate} 

\end{defn}

For a language $L$, we write $Mod(L)$ for the class of $L$-structures.  For a sentence $\varphi$, finitary or infinitary, we write $Mod(\varphi)$ for the class of models of $\varphi$.  For computable structures, satisfaction of computable $\Sigma_\alpha$, or computable $\Pi_\alpha$, formulas is $\Sigma^0_\alpha$, or $\Pi_\alpha$, with all possible uniformity \cite{AK}.  For a computable language $L$, $I(Mod(L))$ is $\Pi^0_2$.  For a computable $\Sigma_\alpha$ $L$-sentence $\varphi$, $I(Mod(\varphi))$ has $\Sigma^0_\alpha$ intersection with $I(Mod(L))$.  Similarly, for a computable $\Pi_\alpha$ $L$-sentence $\varphi$, $I(Mod(\varphi))$ has $\Pi^0_\alpha$ intersection with $I(Mod(L))$.  In particular, if 
$\mathcal{A}$ has a computable $\Sigma_3$ Scott sentence, then $I(\mathcal{A})$ is $\Sigma^0_3$.  If we can show that $I(\mathcal{A})$ is $m$-complete $\Sigma^0_3$, then we know that there is no simpler Scott sentence.  Similarly, if $\mathcal{A}$ has a Scott sentence that is computable $d$-$\Sigma_2$, the conjunction of a computable $\Sigma_2$ sentence with one that is computable $\Pi_2$, then $I(\mathcal{A})$ is $d$-$\Sigma^0_2$.  If we can show that $I(\mathcal{A})$ is $m$-complete $d$-$\Sigma^0_2$, then this Scott sentence is optimal.    

Scott sentences and index sets have been studied for a number of different kinds of structures \cite{idx1}, \cite{idx2}, \cite{idx3}, \cite{idx4}, \cite{idx5}, \cite{idx6}, \cite{idx7}, \cite{idx8}, \cite{idx9}, \cite{idxsets}, \cite{freegroups}, \cite{DescribingII}.  In \cite{freegroups} and \cite{DescribingII}, as in the current paper, the main goal was to find optimal Scott sentences.  In some of the other papers, such as \cite{idxsets}, the main goal was to calculate the complexity of the index set, but finding an optimal Scott sentence was an essential step.  The following thesis was stated in \cite{idxsets}.

\begin{quote}

For a given computable structure $\mathcal{A}$, to calculate the precise complexity of $I(\mathcal{A})$, we need a good description of 
$\mathcal{A}$, and once we have an ``optimal" description (a Scott sentence), the complexity of $I(\mathcal{A})$ will match that of the description.
\end{quote}

In \cite{idxsets}, there are results assigning to each computable reduced abelian $p$-group $G$ of length less than $\omega^2$ a computable infinitary Scott sentence, and in each case, it is shown that the index set is $m$-complete at the level of complexity matching that of the Scott sentence.  Hence, there can be no simpler Scott sentence.  In \cite{freegroups}, it is shown that the free group $F_n$ on $n$ generators has a Scott sentence that is computable $d$-$\Sigma_2$.  It is also shown that $I(F_n)$ is $m$-complete $d$-$\Sigma^0_2$, so there is no simpler Scott sentence.  For the free group $F_\infty$ on 
$\aleph_0$ generators, there is a computable $\Pi_4$ Scott sentence in \cite{freegroups}.  In \cite{DescribingII}, it is shown that $I(F_\infty)$ is $m$-complete $\Pi^0_4$, so there is no simpler Scott sentence.  In these results, the most obvious Scott sentences were not optimal.  Belief in the thesis above led the authors to keep looking until they found an optimal Scott sentences.         

In this paper, we shall give optimal Scott sentences for some further finitely generated groups, and for some torsion-free abelian groups of finite rank.  We show that for both kinds of groups, there is always a computable $\Sigma_3$ Scott sentence.  For all of the finitely generated groups for which we have precise results, there is a computable $d$-$\Sigma_2$ Scott sentence.  We prove that this is the case for all finitely generated abelian groups and for the dihedral group---both quite different from free groups. For the torsion-free abelian groups, we focus on those of rank $1$---subgroups of the additive group of rationals.  For some computable subgroups of $\mathbb{Q}$, the index set is $m$-complete $\Sigma^0_3$, so our computable $\Sigma_3$ Scott sentence is optimal.  For others, we have a computable $d$-$\Sigma_2$ Scott sentence.  For certain computable subgroups of $\mathbb{Q}$, we can show that the index set is $d$-$\Sigma^0_2$, and we are unable to give a computable $d$-$\Sigma_2$ Scott sentence.\footnote{After the present paper was written, in \cite{KM}, Knight and McCoy showed that the group in question does not have computable $d$-$\Sigma_2$ Scott sentence.} 

The results on finitely generated groups are in Section 2. The results on torsion-free abelian groups of finite rank are in Section 3.  In Section 4, we give some open questions.  For background in computability, see \cite{soare}, and for background on computable structures and computable infinitary formulas, see \cite{AK}.           

\section{Finitely generated groups}

In this section, we consider Scott sentences for certain finitely generated groups.  

\begin{prop}
\label{Sigma_3}

Every computable finitely generated group $G$ has a computable $\Sigma_3$ Scott sentence.

\end{prop}

\begin{proof}

As a Scott sentence for $G$, we take the conjunction of the group axioms, which are finitary $\Pi_1$, and the sentence $(\exists \ol{x})\,[\langle\ol{x}\rangle\cong G\ \&\ (\forall y) \bigvee_w w(\ol{x}) = y]$.  We write $\langle\ol{x}\rangle\cong G$ for the conjunction of the formulas $w(\ol{x}) = e$ and $w(\ol{x}) \not= e$ that are true of a fixed generating tuple $\bar{a}$.  This formula says that $\bar{x}$ generates a copy of $G$. The formula $(\forall y) \bigvee_w w(\ol{x}) = y$ says that all elements of the structure are included in $\langle\bar{x}\rangle$.  
\end{proof}

For all of the computable finitely generated groups that we know, there is a Scott sentence of lower complexity.  In this section, we first recall known results on free groups. Next, we consider finitely generated abelian groups.  Here we appeal to the fundamental theorem of finitely generated abelian groups.  Finally, we consider the infinite dihedral group, an infinite group that is neither free nor abelian.  

\subsection{Free groups}

In this subsection, we recall results from \cite{freegroups}. For $n$ finite and $n > 1$, it was shown that the free group on $n$ generators, denoted by $F_n$, has a computable $d$-$\Sigma_2$ Scott sentence, and the index set $I(F_n)$ is $m$-complete $d$-$\Sigma^0_2$.  The Scott sentence for $F_n$ is based on old results of Nielsen, characterizing the orbit under isomorphisms of a basis for $F_n$.  We state Nielsen's results briefly below.  For a more complete account, see the classic text of Lyndon and Schupp \cite{lyndonschupp}.

\begin{defn} [Elementary Nielsen transformation]

Let $(a_1, \ldots, a_n)$ be a basis for $F_n$.  The \emph{elementary Nielsen transformations} are as follows:  

\begin{enumerate}

\item Permute the $a_i$'s.

\item Replace some $a_i$ by $a_i^{-1}$.

\item Replace some $a_i$ by $a_ia_j$, for some $j \ne i$.

\end{enumerate}

\end{defn}

\begin{defn} [Nielsen transformation]

A \emph{Nielsen transformation} is a finite composition of elementary Nielsen transformations.  Two tuples $\ol{a}$ and $\ol{b}$ are said to be \emph{Nielsen equivalent} if there is a Nielsen transformation $\ol{w}(\bar{x})$ such that $\bar{w}(\ol{a}) = \bar{b}$.

\end{defn}

Nielsen showed the following.   

\begin{thm} [Nielsen]
\label{bases}

Any two bases of $F_n$ are Nielsen equivalent.

\end{thm}


\begin{defn} [Primitive tuple]

For an $n$-tuple of variables $\bar{x}$, and $k \le n$, a $k$-tuple of words $\bar{w}(\bar{x})$ is said to be \emph{primitive} if it can be extended to a basis for $F_n$. In particular, if $k = n$, then $\bar{w}(\bar{x})$ represents a Nielsen transformation; i.e., if $\bar{a}$ is a basis for $F_n$, then $\bar{w}(\bar{a})$ is another basis. Otherwise, $\bar{w}(\bar{x})$ is \emph{imprimitive}.

\end{defn} 

It is easy to see that the set of $n$-tuples of words $\bar{w}(\bar{x})$ representing Nielsen transformations is computably enumerable.  It is less obvious that the set of $n$-tuples of words $\bar{w}(\bar{x})$ that are \emph{not} Nielsen transformations is computably enumerable, but Nielsen showed that this is also true.  

\begin{thm} [Nielsen]
\label{prim}

There is an algorithm to determine whether any given $n$-tuple of words $\ol{w}$ is primitive.

\end{thm} 

\begin{proof} [Proof sketch]

By Theorem \ref{bases}, if $\ol{w}$ is a Nielsen transformation, then it can be obtained as a composition of elementary Nielsen transformations. Furthermore, we can take the number of these elementary Nielsen transformations to be at most equal to the total word length of $\ol{w}$. 
\end{proof}

The following theorem from \cite{freegroups} uses the above result on Nielsen transformations.

\begin{thm} [Carson-Harizanov-K-Lange-McCoy-Morosov-Quinn-Safranski-Wallbaum]
Let $F_n$ denote the free group of rank $n$. Then $F_n$ has a computable $d$-$\Sigma_2$ Scott sentence. Furthermore, $I(F_n)$ is $m$-complete $d$-$\Sigma^0_2$.
\end{thm}

\begin{proof} [Proof sketch]

Let $N$ denote the set of all primitive $n$-tuples of words.  Then consider the computable $d$-$\Sigma_2$ sentence that is the conjunction of the following:
\begin{enumerate}

\item the group axioms,

\item  a computable $\Pi_2$ sentence saying that every tuple is generated by an $n$-tuple, 

\item  a computable $\Sigma_2$ sentence saying that there exists an $n$-tuple $\ol{x}$, with no non-trivial relations, such that for any $n$-tuple $\ol{y}$ and any $n$-tuple of words $\ol{w} \notin N$, 
$\ol{w}(\ol{y})\not=\ol{x}$. 

\end{enumerate}

This sentence is in fact a Scott sentence for $F_n$. See \cite{freegroups} for the proof of this fact, along with the proof that the index set of $F_n$ is $d$-$\Sigma^0_2$-hard.
\end{proof}  

\subsection{Generalizing results for free groups}   

If $G$ is a group of rank $n$, it is clear that any $n$-tuple of words representing a Nielsen transformation takes one generating tuple to another.  There may be other tuples of words that also take one $n$-tuple to another.  We have already noted that the following property is true of the finitely generated free groups.               

\begin{perty}[Nielsen Uniqueness Property] 
\label{p1}

All generating tuples of minimal size are Nielsen equivalent.

\end{perty}

A group satisfying Property \ref{p1} has only one equivalence class of minimal generating sets. It may be helpful to think of this property in the following way.  Let $G$ be a group of rank $n$, with presentation $\langle\ol{a}, \bar{R}\rangle$.  Let $f : F_n \rightarrow G$ be a homomorphism onto $G$. Let $\ol{c}$ be a basis for $F_n$ with $f(\ol{c}) = \ol{a}$.  Property \ref{p1} says that if $\ol{b}$ is another $n$-tuple generating $G$, then there is a corresponding basis $\ol{d}$ for $F_n$ with $f(\ol{d}) = \ol{b}$ such that $\ol{a}$ and $\ol{b}$ are Nielsen equivalent.  If $\ol{w}$ is a Nielsen transformation carrying that $\ol{c}$ to $\ol{d}$, then in $G$, via the relators $R$, we have $\ol{w}(\ol{a}) = \ol{b}$.  

Kapovich and Schupp \cite{KS} have shown that a randomly chosen finitely generated group with a single relator has Property \ref{p1}.  Here is a precise statement of their result.     

\begin{thm} [Kapovich-Schupp]

Let $N(n,s)$ be the number of group presentations with $n$ generators $a_1,\ldots,a_n$, and with a single relator $R$ of length at most $s$.  Let $P(n,t)$ be the number of these presentations for which the group has Property \ref{p1}.  Then $\displaystyle\lim_{t\rightarrow\infty} \frac{P(n,s)}{N(n,s)} = 1$.

\end{thm}

In the free group $F_n$, the only $n$-tuples of words that can take one basis to another are the ones that represent Nielsen transformations.  For other groups $G$, an $n$-tuple of words representing a Nielsen transformation takes one minimal-size generating tuple $\bar{a}$ to another.  There may be further $n$-tuples of words that also do this.  For example, if $G$ has the presentation $\la a,b \; | \; aba^{-1}b^{-1} = 1\ra$ (i.e., $G$ is the free abelian group of rank $2$), then the pair of words $(bab,b)$ is not obtained from $(a,b)$ by applying a Nielsen transformation.  However, in $G$, $bab = ab^2$, and $(ab^2,b)$ is obtained from $(a,b)$ by applying a Nielsen transformation, so it is also a generating tuple.

\begin{defn}[Primitive tuple relative to a presentation]

Let $G$ be a group with presentation $\la\bar{a},\bar{R}\ra$, 
where $\bar{a}$ is a generating tuple of minimal length 
$n$.  We say that an $n$-tuple of words 
$\bar{w}$ is \emph{primitive} (for the given presentation) if $\bar{w}$ takes 
$\bar{a}$ to a tuple that is automorphic to $\bar{a}$. That is, $\ol{w}(\ol{a})$ must be in the orbit of $\ol{a}$.

\end{defn}

In the next two subsections, we show that computable finitely generated abelian groups all have computable 
$d$-$\Sigma_2$ Scott sentences, and so does the infinite dihedral group.  It is not known whether every finitely presented group has a $d$-$\Sigma_2$ Scott sentence.  We conjecture that the typical group of the kind considered by Kapovich and Schupp has a $d$-$\Sigma_2$ Scott sentence.          

\begin{conj}

Let $N(n,s)$ be the number of group presentations with $n$ generators $a_1,\ldots,a_n$, and with a single relator $R$ of length at most $s$.  Let $P(n,s)$ be the number of these presentations for which the group has Property \ref{p1}.  Then $lim_{s\rightarrow\infty} \frac{P(n,s)}{N(n,s)} = 1$.   

\end{conj}

\subsection{Finitely generated abelian groups}

Apart from $\Z$, the finitely generated abelian groups are not free. Nevertheless, they all have computable $d$-$\Sigma_2$ Scott sentences.

\begin{thm}
\label{abcomp}

If $G$ is a finitely generated abelian group, then $G$ has a computable $d$-$\Sigma_2$ Scott sentence. Furthermore, $I(G)$ is $m$-complete $d$-$\Sigma^0_2$ for infinite $G$.

\end{thm}

Let $G$ be an infinite, finitely generated abelian group. We first prove that $G$ has a computable $d$-$\Sigma_2$ Scott sentence. Then we show that $I(G)$ is $d$-$\Sigma^0_2$-hard.  This implies that the computable $d$-$\Sigma_2$ Scott sentence is optimal.

\begin{proof}

To formulate a $d$-$\Sigma_2$ description of $G$, we invoke the well-known fundamental theorem of finitely generated abelian groups, which states that every such group can be expressed as a direct sum of cyclic groups. In particular, $G$ is isomorphic to $\Z^n \oplus T$ for some 
$n > 0$ and some finite abelian group $T$. We first give a Scott sentence for $\Z^n$.  

\begin{lem}

We take the conjunction of the abelian group axioms and the following sentences to obtain a Scott sentence for $\Z^n$.

\begin{enumerate}

\item $(\forall x) [ x \ne 0 \rightarrow \displaystyle\bigdoublewedge_{n > 0} nx \ne 0]$

\item $(\exists x_1) \cdots (\exists x_n) \left[ \left( (\forall y) \displaystyle\bigdoublewedge_{1 \le i \le n; \; k > 0} ky \ne x_i \right) \wedge \displaystyle\left(\bigdoublewedge_{k_i \emph{\text{ not all }} 0} k_1x_1 + \cdots + k_nx_n \ne 0 \right) \right]$

\item $(\forall x_1) \cdots (\forall x_{n+1})\displaystyle\bigdoublevee_{k_i \emph{\text{ not all }} 0} k_1x_1 + \cdots + k_{n+1}x_{n+1} = 0$
\end{enumerate}
\end{lem}

Sentence (1) says that $\Z^n$ is torsion-free. Sentence (2) asserts the existence of $n$ linearly independent elements not divisible by any $k$.  Finally, Sentence (3) says that every set of $n+1$ elements is linearly dependent.  It is clear that the conjunction describes $\mathbb{Z}^n$ up to isomorphism.  We may also use the sentence to describe the torsion-free component of $G$. 

\bigskip  

We now focus on the torsion component of $G$.  This is a finite group.  We recall a result from \cite{idxsets}.  

\begin{thm}[Calvert-Harizanov-K-Miller]
\label{finstruct}

Let $L = (R_1, \ldots, R_\ell)$ be a finite relational language, and let $\mathcal{A}$ be a finite $L$-structure.  Then $\mathcal{A}$ has a finitary Scott sentence that is the conjunction of an existential sentence and one that is universal.  

\end{thm}

\begin{proof} [Proof]

Let $k$ be the cardinality of $\mathcal{A}$.  We take the conjunction of the following sentences:

\begin{enumerate}

\item $(\exists x_1) \cdots (\exists x_k) \; \delta(x_1, \ldots, x_k)$, where $\delta$ is the conjunction of sentences in $D(\mathcal{A})$

\item $(\forall x_1) \cdots (\forall x_{k+1}) \displaystyle\bigvee_{i \ne j} x_i = x_j$ 

\end{enumerate}
Sentence (1) guarantees that there is a substructure isomorphic to $\mathcal{M}$, and sentence (2) says that the structure does not have $k+1$ distinct elements.  \end{proof}

\noindent
\textbf{Remark}.  The language of groups is not relational.  Nonetheless, Theorem \ref{finstruct} holds for any finite group $T$.  What is important is that we have a finitary quantifier-free formula $\delta_T(\bar{x})$ guaranteeing that $\bar{x}$ forms a group isomorphic to $T$.     

\bigskip

We return to the finitely generated abelian group $G \cong \mathbb{Z}^n \oplus T$, where $T$ is finite. 

\begin{lem}
\label{abdesc}

Let $G$ be a finitely generated abelian group.  Say $G \cong \mathbb{Z}^n \oplus T$, where $T$ is the torsion part.  Let $\delta_T(\bar{x})$ be a finitary quantifier-free formula guaranteeing that $\bar{x}$ is a copy of $T$.  Then we obtain a Scott sentence for $G$ by taking the conjunction of the axioms for abelian groups and the following sentences:

\begin{enumerate}

\item $(\exists x_1) \cdots (\exists x_k) (\forall x) \left(\delta_T(x_1, \ldots, x_k) \wedge \displaystyle\left(\left(\displaystyle\bigdoublewedge_{m > 0} (mx \ne 0)\right) \vee \bigvee_{1 \le i \le k} x = x_i\right)\right)$

\item $(\exists x_1) \cdots (\exists x_n) \left[ \left( (\forall y) \displaystyle\bigdoublewedge_{1 \le i \le n, k > 0} ky \ne x_i \right) \wedge \displaystyle\left(\bigdoublewedge_{k_i \emph{\text{ not all }} 0} k_1x_1 + \cdots + k_nx_n \ne 0 \right) \right]$

\item $(\forall x_1) \cdots (\forall x_{n+1})\displaystyle\bigdoublevee_{k_i \emph{\text{ not all }} 0} k_1x_1 + \cdots + k_{n+1}x_{n+1} = 0$
\end{enumerate}
\end{lem}

\begin{proof} 

Sentence (1) states that every element is either in $T$ or has infinite order. Sentences (2) and (3) are as in the Scott sentence for $\Z^n$, since they remain true for groups with torsion elements. The above sentences characterize $G$ up to isomorphism.
This finishes the proof that every finitely generated abelian group has a computable $d$-$\Sigma_2$ Scott sentence.  
\end{proof}

Next, we show that the computable $d$-$\Sigma_2$ sentence is optimal.   

\begin{lem}
\label{abhard}

Let $G$ be a finitely generated abelian group.  Then $I(G)$ is $d$-$\Sigma^0_2$-hard.

\end{lem}

\begin{proof}

We have $G \cong \mathbb{Z}^k \oplus T$, for some $k\geq 1$ and some finite group $T$.  Let $S$ be an arbitrary $d$-$\Sigma^0_2$ set.  We must show that 
$S\leq_m I(G)$.  To do this, we produce a uniformly computable sequence of groups $(\G_n)_{n\in\omega}$ such that $\G_n\cong G$ if and only if $n\in S$.  For simplicity, we suppose that $G \cong \mathbb{Z}^k$.  We also assume that $k > 1$, since the case $G \cong \Z$ is proved in \cite{idxsets}.  Say that $S = S_1 - S_2$, where $S_1$ and $S_2$ are 
$\Sigma^0_2$ sets. We construct $\G_n$ such that 
\[\G_n \cong \left\{ \begin{array}{ll}
\mathbb{Z}^{k-1} & \mbox{if $n \notin S_1$} \\ 
\mathbb{Z}^k & \mbox{if $n \in S_1 - S_2$} \\
\mathbb{Z}^{k+1} & \mbox{if $n \in S_1 \cap S_2$} \end{array} \right. \]

For each $n \in \N$, we construct the diagram $D(\G_n)$ corresponding to $\G_n$ above.  Let $Z_n$ denote the \emph{target structure} for the construction. That is, let $Z_n = \Z^{k-1}$ if $n \in S_1$, $Z_n = \Z^k$ if $n \in S_1 - S_2$, and $Z_n = \Z^{k+1}$ if 
$n \in S_1 \cap S_2$.  

We recall that for a $\Sigma^0_2$ set $A$, there is a uniformly computable sequence of computable approximations $(A_s)_{s\in\omega}$ such that $n\in A$ if and only if there is some $s_0$ such that for all $s\geq s_0$, $n\in A_s$.  For the $\Sigma^0_2$ sets $S_1$ and $S_2$, we take the standard computable approximation sequences 
$S_{1, s}$ and $S_{2, s}$.  At stage $s$, we believe that $n\in S_i$ just in case $n\in S_{i,s}$.  We will construct an isomorphism $f_n : \G_n \rightarrow Z_n$.  
At stage $s$, we have a partial isomorphism $f_{n, s}$ from $\G_n$ to a structure $Z_{n, s}$, representing our stage $s$ guess at the target structure, where this changes among the three possibilities above exactly as our belief of $n \in S_i$ changes.  Also, at stage $s$, we determine a finite part $d_{n,s}$ of the diagram of 
$\mathcal{G}_n$ such that all constants appearing in $d_{n,s}$ are in the domain of $f_{n,s}$, and $f_{n,s}$ interprets these constants so as to make the sentences of $d_{n,s}$ true.  Below, we say precisely how $f_{n, s}$ is determined.  We must have $d_{n,s}\subseteq d_{n,s+1}$.  For this, we use the following.  

\bigskip
\noindent
\textbf{Fact}.  Let $\bar{c}$ be a generating tuple for $\mathbb{Z}^m$, and let $\bar{c},c'$ generate $\mathbb{Z}^{m+1}$.  Then any existential formula true of $\bar{c}$ in $\mathbb{Z}^{m+1}$ is also true of $\bar{c}$ in $\mathbb{Z}^m$.        

\begin{proof} [Proof of Fact]

We could prove this using results of Szmielew \cite{Szmielew}.  We can also give a direct proof as follows.  There is an elementary extension  $G$ of $\mathbb{Z}^m$ with an element $d$ that is linearly independent of $\bar{c}$.  Let $G'$ be the subgroup of $G$ generated by $\bar{c},d$.  Then 
$(G',\bar{c},d)\cong(\mathbb{Z}^{m+1},\bar{c},c')$.  Any existential formula true of $\bar{c}$ in $\mathbb{Z}^{m+1}$ is true in $G'$, so it is true in $G$, and in 
$\mathbb{Z}^m$.
\end{proof}     

Let $C$ be an infinite computable set of constants.  For all $n$, this will serve as the universe of $\mathcal{G}_n$, and the domain of $f_n$. We start by assuming that $n \notin S_{1, 0}$, so that at stage $0$, our target structure is $Z_{n, 0} = \mathbb{Z}^{k-1}$.  We designate initial generators $a_1, \ldots, a_{k-1} \in C$.  Passing from stage $s$ to $s+1$, if we see that $n \in S_{i, s}$ and $n \in S_{i, s+1}$ (or, respectively, if $n \notin S_{i, s}$ and $n \notin S_{i, s+1}$), then our belief of $n \in S_i$ has not changed, and we extend $f_{n,s}$ to $f_{n,s+1}$, with the same target structure.  However, if our guess at whether $n \in S_i$ changes, then the target structure also changes.  We consider the different kinds of changes that may occur.  

First, suppose $n \notin S_{1, s}$ but $n \in S_{1, s+1}$. Then we add a new generator $a$ that is independent of the current generators $a_i$ and extend $f_{n,s}$ accordingly, using the target structure $Z_{n, s+1} = \Z^k$.  On the other hand, if we see $n \in S_{1, s}$ but $n \notin S_{1, s+1}$, our target structure falls back to $Z_{n, s+1} = \Z^{k-1}$.  In this case, $f_{n,s+1}$ extends $f_{n,s'}$ for the greatest $s' < s$ such that $n\notin S_1$.  By the fact above, we can express the extra generator as a linear combination of the others.  Now, suppose our target structure changes because of $S_2$.  As in the case of $S_1$, we either introduce or collapse a generator.  Say $n \in S_{1, s+1}$, and $n \notin S_{2, s}$, but $n \in S_{2, s+1}$.  Then we add a new element $a$, independent of the current generators $a_i$, and we extend $f_{n,s}$ accordingly.  If for some $s'' > s$, we see $n \in S_{2, s''}$ but $n \notin S_{2, s''+1}$, then we collapse $a$.  Then $f_{n,s''+1}$ extends $f_{n,s'}$ for the greatest $s' < s''$ such that $n\in S_{1,s'} - S_{2,s'}$.    

Given the partial isomorphisms $(f_s)_{s \in \omega}$, we build $f : \G_n \rightarrow Z_n$.  If $n \notin S_1$, then we take the union of all $f_{n,s}$ such that 
$n \notin S_{1, s}$. If $n \in S_1 - S_2$, then for some $s_1$, $n \in S_{1, s}$ for all $s \ge s_1$, so we take the union of $f_{n,s}$ for $s \ge s_1$ such that $n \notin S_{2, s}$. If $n \in S_1 \cap S_2$, then for some $s_2$, $n \in S_{2, s}$ for all $s \ge s_2$, so we take the union $f_{n,s}$ for $s \ge s_2$. This gives us the required isomorphism between $\G_n$ and $Z_n$, and therefore gives a uniform way of computing the $\G_n$. So $S$ is reducible to $I(G)$.  This completes the proof of Lemma \ref{abhard}.  
\end{proof}

By Lemma \ref{abdesc}, $G$ has a computable $d$-$\Sigma_2$ Scott sentence, so $I(G)$ is $d$-$\Sigma^0_2$.  By Lemma \ref{abhard}, $I(G)$ is $m$-complete $d$-$\Sigma^0_2$.  This completes the proof of Theorem \ref{abcomp}.
\end{proof}

\subsection{The infinite dihedral group} 

In this subsection, we look at an example of a group that is neither free nor abelian, but still has a computable 
$d$-$\Sigma_2$ Scott sentence. 

\begin{defn}

The infinite dihedral group $D_\infty$ is defined by the standard presentation $\langle a, b \; | \; a^2 = b^2 = 1 \rangle$. 

\end{defn}

We are grateful to the referee for providing a reference to work of Benois \cite{Benois}, showing the following. We give a proof below, because the reference \cite{Benois} is somewhat difficult to find, and because it is relevant to the following proofs. We apply the result of Benois to show that $D_\infty$ has a computable $d$-$\Sigma_2$ Scott sentence.

\begin{prop} [Benois]
\label{Benois}

Let $G$ be a copy of $D_\infty$ with presentation $\la a,b \; | \; a^2 = b^2 = 1\ra$.  Then we can effectively determine whether a given pair of words $\bar{w}$ takes $(a,b)$ to a generating pair.      

\end{prop}

\begin{proof}

Recall that $\bar{w}$ is primitive if it takes $(a,b)$ to a pair in the orbit of $(a,b)$; i.e., $\bar{w}(a,b)$ is a generating pair satisfying exactly the same relations as $(a,b)$. Before starting the proof, we recall regular expressions (see Sipser \cite{sipser}). Let $S$ be any set.  Then a regular expression is a syntactic object that represents a set of strings in $S$.  We define \emph{regular expressions} on $S$, and the corresponding sets of strings, inductively, as follows:

\begin{enumerate}

\item $\epsilon$ is a regular expression, called the \emph{empty expression}.  The corresponding set of strings is $\emptyset$.

\item For each $s \in S$ there is a regular expression $x_s$, with corresponding set of strings given by $\{s\}$. It is common to simply write $s$ instead of $x_s$.

\item If $x$ and $y$ are regular expressions, then the \emph{concatenation}, $xy$, is a regular expression.  The corresponding set of strings consists of the strings $vw$, where $v$ is in the set corresponding to $x$ and $w$ is in the set corresponding to $w$.  

\item If $x$ and $y$ are regular expressions, then the \emph{union}, $x \cup y$, is a regular expression.  The corresponding set of strings is the union of the set corresponding to $x$ and that corresponding to $y$.   

\item If $x$ is a regular expression, then the \emph{Kleene star}, $x^*$, is a regular expression.  The corresponding set of strings is the set of all finite concatenations of strings in the set corresponding to $x$.      

\end{enumerate}
In practice, if $x$ is a regular expression, and $w$ is a word on $S$, we say \emph{$w$ has form $x$}, or \emph{$w\in x$}, if $w$ is in the set of strings corresponding to $x$.

The relations $a^2 = b^2 = 1$ guarantee that $a^{-1} = a$ and $b^{-1} = b$.  We begin the proof with the assumption that any word representing an element of $D_\infty$ has form $(ab)^*(a \cup \epsilon)$ or $(ba)^*(b \cup \epsilon)$, since any consecutive pair $aa$ or $bb$ may be eliminated. This significantly restricts the form a word may take: a word $w$ in $D_\infty$ is now determined by the last letter of $w$, and the length of $w$.  We start with the base cases:
\begin{enumerate}
\item $(w_1', \epsilon)$ is not a generating pair,

\item a pair of the form $(a(ba)^*, b(ab)^*)$ is not a generating pair unless it is $(a, b)$,

\item $(a, b)$ is a generating pair.

\end{enumerate}
Pairs satisfying (1) and (2) are not generating pairs because they do not generate $D_\infty$, except for the pair $(a,b)$ itself.

We prove that every pair of words is Nielsen equivalent to one of these cases. With these base cases, we use induction on the total word length of a pair.  Consider a pair of words $(w_1, w_2)$ such that $|w_1| + |w_2| = n$, and inductively assume that every pair of words $(w_1', w_2')$ such that $|w_1'| + |w_2'| < n$ is equivalent to one of the base cases above. We show that either $(w_1, w_2)$ \emph{is} one of the base cases, or we may find a pair $(w_1', w_2')$ that is equivalent to 
$(w_1, w_2)$, yet strictly shorter in total word length.

Suppose $(w_1, w_2)$ is not one of the above base cases, and without loss of generality, suppose $|w_1| \ge |w_2|$. Again, without loss of generality, we suppose that $w_1$ starts with $a$. We now consider two cases: $w_2$ starts with $a$, and $w_2$ starts with $b$. Suppose $w_2$ starts with $a$. Then $w_2$ is an initial segment of $w_1$. Also, observe that $w_2^{-1}$ is simply the reverse of $w_2$. Therefore $(w_1, w_2^{-1}w_1)$ is an equivalent pair that is strictly shorter in total word length.  Now, we suppose that $w_2$ starts with $b$. We have two further sub-cases: $w_1$ does not end in $a$, or $w_2$ does not end in $b$ (both may be true). Suppose $w_2$ does not end in $a$, in which case it ends in $b$. Then $(w_1w_2, w_2)$ is shorter than $(w_1, w_2)$. Suppose instead that $w_1$ does not end in $b$, in which case it ends in $a$. Then $(w_2w_1, w_2)$ is shorter than $(w_1, w_2)$.  

Thus, by induction, any $(w_1, w_2)$ is equivalent to one of the base case forms. This means we can keep shortening our pair of words until we reach a base case, which gives us an algorithm to do so. This in turn gives us an algorithm to decide whether $(w_1(a, b), w_2(a, b))$ is a generating pair. 

\end{proof}

\begin{lem}
\label{lem2}

The only generating pairs $(x_1,x_2)$ that satisfy $x_1^2 = x_2^2 = 1$ are $(a,b)$ and $(b,a)$.  

\end{lem}

\begin{proof}

As in the proof of Proposition \ref{Benois} above, if $(x_1,x_2)$ is a generating pair that is not simply $(a,b)$ or $(b,a)$, then one of $x_1$ or $x_2$ must have a reduced form of either $a(ba)^*b$ or $b(ab)^*a$.  The square of such an $x_i$ is not equal to $1$, therefore $x_i^2 \ne 1$, as required. On the other hand, it is clear that $(a, b)$ and $(b, a)$ satisfy the above equality.    
\end{proof}

A pair of words $\bar{w}(x_1,x_2)$ is primitive (with respect to the standard presentation of $D_\infty$) if $\bar{w}(a,b)$ is in the orbit of $(a,b)$, and otherwise, it is imprimitive.  We can effectively determine, for a pair of words $\bar{w}(a,b)$, whether it gives a generating pair, and whether it is equal, after reduction, to $(a,b)$ or $(b,a)$.  Thus, the set of primitive pairs of words, and the set of imprimitive pairs of words, are both computable.           

\begin{prop}

There is a computable $d$-$\Sigma_2$ Scott sentence for $D_\infty$.

\end{prop}

\begin{proof}

Let $\sigma$ be the conjunction of a computable $\Pi_1$ sentence $\sigma_1$ axiomatizing groups, a computable 
$\Pi_2$ sentence $\sigma_2$ saying that every triple is generated by a pair $(x_1,x_2)$ such that 
$x_1^2 = x_2^2 = 1$, and a computable $\Sigma_2$ sentence $\sigma_3$ saying that there is a pair $(x_1,x_2)$ satisfying exactly the relations satisfied by $(a,b)$ and such that for all $(y_1,y_2)$ satisfying $y_1^2 = y_2^2 = 1$, and for all imprimitive pairs of words $\bar{w}$, $\bar{w}(y_1,y_2)\not= (x_1,x_2)$.  We show that $\sigma$ is a Scott sentence.         

First, we show that $D_\infty$ satisfies $\sigma$.  Clearly, $D_\infty$ satisfies $\sigma_1$ and $\sigma_2$.  For $\sigma_3$, we let $(x_1,x_2)$ be $(a,b)$.  For a pair $(y_1,y_2)$, if there is a pair of words $\bar{w}$ such that $\bar{w}$ takes $(y_1,y_2)$ to $(a,b)$, then $(y_1,y_2)$ is a generating pair.  If $y_1^2 = y_2^2 = 1$, then by Lemma \ref{lem2}, $(y_1,y_2)$ must be either $(a,b)$ or $(b,a)$, so the pair of words $\bar{w}(y_1,y_2)$ cannot be imprimitive.  Therefore, $D_\infty$ satisfies $\sigma$.      

Next, we show that any countable group $G$ that satisfies $\sigma$ is isomorphic to $D_\infty$.  Take $(a,b)$ witnessing that $G$ satisfies $\sigma_3$.  Then the subgroup $H$ of $G$ generated by $(a,b)$ is isomorphic to $D_\infty$.  Now take any $c$ in $G$.  Since $G$ satisfies $\sigma_2$, the triple $(a,b,c)$ is generated by a pair $(a',b')$ such that $(a')^2 = (b')^2 = 1$.  Say that $\bar{w}$ is a pair of words such that $\bar{w}(a',b') = (a,b)$.  Since $G$ satisfies $\sigma_3$, it follows that $\bar{w}$ is primitive.  Then $\bar{w}(a,b)$ is equal to $(a,b)$ or $(b,a)$, and the equality is proved just from the group axioms and the relations $a^2 = b^2 = 1$.  This means that $\bar{w}(a',b')$ is either $(a',b')$ or $(b',a')$, so $(a',b') = (a,b)$ or $(b',a') = (a,b)$.  From this, it is clear that $H = G$.  This completes the proof that $\sigma$ is a Scott sentence.        
\end{proof}

We see that $D_\infty$ has a computable $d$-$\Sigma_2$ Scott sentence.  We can show that this is optimal.  

\begin{lem}
\label{dihhard}

$I(D_\infty)$ is $d$-$\Sigma^0_2$-hard.

\end{lem}

\begin{proof}

We show that for any $d$-$\Sigma^0_2$ set $S$, $S\leq_m I(D_\infty)$.  Let $S = S_1 - S_2$, where $S_1$ and $S_2$ are $\Sigma^0_2$ sets.  Let $H$ be an infinitely generated group, with elements $\{a_i\}_{i \in \omega}$ and $b$ such that $a_{i+1} = a_iba_i$.  Uniformly in $n$, we will enumerate the atomic diagram of a copy of $H$, if $n\notin S_1$ and a copy of $D_\infty$ if $n\in S_1 - S_2$.  If $n\in S_1\cap S_2$, then we will enumerate only finitely much of the diagram of a copy of $D_\infty$.   
Let $S_{1, s}$ and $S_{2, s}$ be computable approximations for $S_1$ and $S_2$.

Without loss of generality, we may assume $n \in S_{1, 0} - S_{2, 0}$, so at stage $0$, our target set is $D_\infty$. We start with generators $a$ and $b$. At stage $s+1$, if there is no change to the target structure, then we continue to build the structure according to the diagram of the current target. However, if we change our belief of whether $n \in S_i$, we must change our target set accordingly.

First suppose we change our structure because of $S_1$.  Say $n \notin S_{1, s}$ and $n \in S_{1, s+1}$. Then we continue building our isomorphism, without having to collapse any elements.  On the other hand, if we see $n \in S_{1, s}$ but $n \notin S_{1, s+1}$, then we express $a$ as $a'ba'$, and we think of $a'$ and $b$ as our new generating pair. Notice that if infinitely often we come back to believing that $n \notin S_1$, we will produce a group that is not finitely generated, and thus not isomorphic to $D_\infty$. In this case we construct the diagram of a copy of $H$. If $n \in S_1 - S_2$, then we will eventually stop returning to $H$ and we will have a group isomorphic to $D_\infty$.

Now suppose our structure changes because of $S_2$. This time, we modify our isomorphism by either stopping or continuing our construction of a copy of $D_\infty$. Say $n \notin S_{2, s}$ but $n \in S_{2,s+1}$. Then we stop adding to the diagram. If we later see $n \in S_{2, s}$ but $n \notin S_{2, s+1}$, then we continue enumerating the diagram of a copy of $D_\infty$ where we left off. If at infinitely many stages, we believe that $n \in S_1 - S_2$, then we will end up with the atomic diagram of a copy of $D_\infty$. If we never return to believing $n \in S_1 - S_2$, then $n \in S_1 \cap S_2$, and we will be left with a finite fragment of the diagram.  This shows that $S\leq_m I(D_\infty)$, so $I(D_\infty)$ is $d$-$\Sigma^0_2$-hard.
\end{proof}

Our results on $D_\infty$ are summarized in the following theorem.

\begin{thm}
\label{dihcomp}

$D_\infty$ has a computable $d$-$\Sigma_2$ Scott sentence, and $I(D_\infty)$ is $m$-complete $d$-$\Sigma^0_2$.

\end{thm}

\section{Torsion-free abelian groups}

The torsion-free abelian groups are precisely the subgroups of vector spaces over $\mathbb{Q}$. The \emph{rank} of such a group is the least dimension of a vector space in which the group can be embedded.  In \cite{idxsets}, there are Scott sentences and index sets for the $\mathbb{Q}$-vector spaces themselves.  Most torsion-free abelian groups of finite rank are not finitely generated.  However, we have the analogue of Proposition \ref{Sigma_3}.           

\begin{prop}
\label{sigma3desc}

Let $G$ be a computable torsion-free abelian group of finite rank. Then $G$ has a computable $\Sigma_3$ Scott sentence.  

\end{prop}

\begin{proof}

Suppose $G$ has rank $n$.  We have $n$ elements $x_1, \ldots, x_n$ such that every element of $G$ is uniquely represented as a linear combination of elements of $x_1, \ldots, x_n$.  Let $\Lambda$ be the set of linear combinations that actually occur.  We have a Scott sentence that is the conjunction of the axioms for torsion-free abelian groups, which are finitary $\Pi_1$ and the computable $\Sigma_3$ sentence
\[(\exists x_1) \cdots (\exists x_n) [\bigdoublewedge_{\lambda\in\Lambda}(\exists y) y = \lambda(x_1, \ldots, x_n)\ \wedge \ (\forall y)\bigdoublevee_{\lambda\in\Lambda}y = \lambda(x_1, \ldots, x_n)]\ .\]  
\end{proof}

We focus on torsion-free abelian groups of rank $1$.  These are the additive subgroups of $\Q$.  For more about these groups, see \cite{BZ}.  In some cases, we can show that the index set is $m$-complete 
$\Sigma^0_3$, so the computable $\Sigma_3$ description in Proposition \ref{sigma3desc} is optimal.  In other cases, there is a computable $d$-$\Sigma_2$ Scott sentence, and we show that this is the best possible description.  

We can describe a subgroup $G$ of $\mathbb{Q}$ by fixing an element to play the role of $1$, and saying which prime powers divide this element.  
Let $G$ be a computable subgroup of $\mathbb{Q}$, and let $P$ denote the set of primes.  We distinguish among cases by partitioning $P$ into sets $P^0$, $P^{fin}$, and $P^\infty$, where:

\begin{enumerate}

\item $P^0 = \{p \in P : G \models p \nmid 1\}$

\item $P^{fin} = \{p \in P : G \models p^k \mid 1\text{\emph{ and }}p^{k+1} \nmid 1\text{\emph{ for some }} k > 0 \}$

\item $P^\infty = \{p \in P : G \models p^k \mid 1\text{\emph{ for all }} k > 0\}$
\end{enumerate}

Notice that $P^0$ is $\Pi^0_1$ and $P^{fin} \cup P^\infty$ is $\Sigma^0_1$.  Also, $P^\infty$ is $\Pi^0_2$ and $P^{fin}$ is $\Sigma^0_2$. 

\subsection{Hardness}  

In this subsection, we give some hardness results on the index sets of subgroups $G$ of $\Q$.

\begin{lem}
\label{dsigma2hard}

Suppose $G$ is a computable subgroup of $\mathbb{Q}$ such that $P^0 \cup P^{fin} \ne \emptyset$ and $P^\infty \ne \emptyset$. Then $I(G)$ is 
$d$-$\Sigma^0_2$-hard.   

\end{lem}

\begin{proof}

Fix primes $p \in P^0 \cup P^{fin}$ and $q \in P^\infty$, and let $k$ be greatest such that $G\models p^k|1$ . We must show that for an arbitrary $d$-$\Sigma^0_2$ set $S$, $S\leq_m I(G)$.  Let $S = S_1 - S_2$, where $S_1$ and $S_2$ are $\Sigma^0_2$.  We describe a uniformly computable sequence of groups 
$(\mathcal{G}_n)_{n\in\omega}$ such that $\mathcal{G}_n \cong G$ if and only if $n\in S$.  We use alternative structures $H$ and $K$, where $H\subseteq\mathbb{Q}$ is the smallest extension of $G$ in which $p$ divides $1$ infinitely, and $K$ is the largest subgroup of $G$ in which $q$ does not divide $1$.  We will arrange that 

\[ \mathcal{G}_n \cong \left\{ \begin{array}{ll}
H & \mbox{if $n \notin S_1$} \\ 
G & \mbox{if $n \in S_1 - S_2$} \\
K & \mbox{if $n \in S_1 \cap S_2$} \end{array} \right. \]  
Let $S_{1, s}$ and $S_{2, s}$ be computable approximations for $S_1$ and $S_2$, respectively, such that
$n \in S_i$ if and only if for all but finitely many $s$, $n \in S_{i, s}$.  

Let $C$ be an infinite computable set of constants. This will be the universe of $\mathcal{G}_n$ for all $n$.  At stage $s$, we determine a finite partial isomorphism $f_{n,s}$ from $\mathcal{G}_n$ to a target structure.  The choice of the target structure at stage $s$ is based on $S_{1,s}$ and $S_{2,s}$.  The target structure is $H$ if $n \notin S_{1,s}$, $G$ if $n \in S_{1,s} - S_{2,s}$, and $K$ if $n \in S_{1,s} \cap S_{2,s}$.  We obtain the desired isomorphism by taking the union of certain $f_{n,s}$.  We suppose that $n \notin S_{1, 0}$, so at stage $0$, our target structure is 
$H$. Fix a $1 \in H$. We enumerate the atomic diagram of $\mathcal{G}_n$ in stages.  At stage $s$, we have a finite part $d_{n,s}$, and we suppose that $f_{n,s}$ maps the constants that appear in $d_{n,s}$ to the target structure so as to make the $d_{n,s}$ true.  Given $f_{n,s}$, and a $d_{n,s}$, we describe what happens at stage $s+1$; in particular, what happens when we change target structures.  

We first say what occurs when we switch between $H$ and $G$. If $n \notin S_{1, s}$ but $n \in S_{1, s+1}$, then we switch from $H$ to $G$.  If $n \in S_{1, s}$ but $n \notin S_{1, s+1}$, then we switch from $G$ to $H$, so that $f_{n,s+1}$ extends $f_t$ for the greatest stage $t < s$ for which the target structure was $H$.  Supposing our target structure is $H$, we keep dividing our designated $1$ by $p$. If we switch to thinking that our structure is actually $G$, we stop dividing by $p$.  We need to change our $1 \in G$, since $1$ is now divisible by $p$. So, we replace $1$ by 
$p^k/p^{k_s}$, where $k_s$ is the highest power of $p$ dividing $1$ at stage $s$. If, at some later stage, we return to $H$ as our target structure, we go back to our original $1 \in H$ and continue dividing it by $p$. In this way, if we return to $H$ infinitely often, then $p$ will divide $1$ infinitely. If we eventually stop having $H$ as the target structure, then we are left with the special element $1$ that is divisible by $p$ only $k$ times.

Now, consider the cases where we switch between structures $G$ and $K$. If $n \notin S_{2, s}$ but $n \in S_{2, s+1}$, then we switch from $G$ to $K$.  While we believe our target structure is $G$, we keep dividing $1$ by $q$.  Suppose that at stage $s$, the target structure is $G$, but at stage $s+1$, it is $K$. Then we stop dividing $1$ by $q$, and replace $1 \in G$ by $1/q^{\ell_s}$, where $\ell_s$ is the highest power of $q$ dividing $1$ at stage $s$. If at a later stage, the target structure is again $G$, then we return to the original $1 \in G$, and we continue dividing $1$ by powers of $q$.  We let $f_{n,s+1}$ extend $f_{n,t}$ for the greatest $t < s$ such that the target structure was $G$.  Thus, if we return to $G$ infinitely often, then $q$ will divide $1$ infinitely, as it should. Otherwise we will eventually settle on $K$, in which case, $q$ does not divide $1$ even once.

Given the sequence $(f_{n,s})_{s \in \omega}$ constructed by this procedure, we obtain an isomorphism from $\mathcal{G}_n$ to the desired structure as follows.  If 
$n \notin S_1$, then $f_n$ is the union of all $f_{n,s}$ such that $n \notin S_{1, s}$.  This is an isomorphism from $\mathcal{G}_n$ to $H$.  If $n \in S_1 - S_2$, there is some $s_1$ such that for all $s\geq s_1$ $n\in S_{1,s}$, and for infinitely many $s\geq s_1$, $n\notin S_{2,s}$.  In this case, $f_n$ is the union of the $f_{n,s}$ for 
$s \ge s_1$ such that $n \notin S_{2, s}$.  This is an isomorphism from $\mathcal{G}_n$ to $G$.  If $n \in S_1 \cap S_2$, there is some $s_2\geq s_1$ such that for all $s\geq s_2$, $n \in S_{2,s}$.  In this case, $f_n$ is the union of $f_{n,s}$ for $s \ge s_2$.  This is an isomorphism from $\mathcal{G}_n$ to $K$.  This concludes the proof.
\end{proof}

Lemma \ref{dsigma2hard} gives conditions guaranteeing that $I(G)$ is $d$-$\Sigma^0_2$-hard.  The lemma below gives conditions guaranteeing that $I(G)$ is 
$\Sigma^0_3$-hard.  We will use the well-known fact that the set $Cof = \{e : W_e\text{ is co-finite}\}$ is $m$-complete $\Sigma^0_3$.  

\begin{lem}
\label{sigma3hard}

Suppose $G$ is a computable subgroup of $\mathbb{Q}$, with $P^{fin}$ having an infinite computable subset $A$.  Then $I(G)$ is $\Sigma^0_3$-hard.   

\end{lem}

\begin{proof} 

To show that $Cof\leq_m I(G)$, we produce a uniformly computable sequence of groups $(\mathcal{G}_n)_{n\in\omega}$ such that $\mathcal{G}_n\cong G$ if and only if $n\in Cof$.  Let $(p_i)_{i \in\omega}$ be the elements of $A$, in order.  As in the previous proof, the universe of all $\mathcal{G}_n$ is an infinite computable set $C$ of constants.  We designate an element of $C$ to play the role of $1$, and we construct $\mathcal{G}_n$ such that 
\begin{enumerate}
\item  for primes $p\notin A$, $\mathcal{G}_n\models p^r\mid1$ if and only if $G\models p^r\mid1$, 
\item  for $p_k\in A$, if $G\models p_k^{r+1} \mid 1$, then 
\begin{enumerate}
\item  $\mathcal{G}_n\models p_k^r \mid1$,
and 
\item  $\mathcal{G}_n\models p_k^{r+1} \mid 1$ if and only $k\in W_n$.  
\end{enumerate}
\end{enumerate}

If $n\in Cof$, then $W_n$ is co-finite. Let $m$ be the product of those $p_k$ for which $k\notin W_n$.  We replace our original $1$ with $m \cdot 1$.  Then our new $1$ is divisible by exactly the same prime powers as the designated $1$ in $G$, so $\mathcal{G}_n\cong G$. If $n\notin Cof$, then, since $W_n$ is co-infinite, there is no element of $\mathcal{G}_n$ divisible by the same prime powers as the designated $1$ in $G$, so $\mathcal{G}_n\not\cong G$.

The existence of a such a sequence $\G_n$ proves that $Cof\leq_m I(G)$. Furthermore, since $Cof$ is known to be $m$-complete $\Sigma^0_3$, we conclude that $I(G)$ is $\Sigma^0_3$-hard.
\end{proof}

\subsection{Completeness}

Suppose $G$ is a computable subgroup of $\mathbb{Q}$.  By Proposition \ref{sigma3desc}, $G$ has a computable $\Sigma_3$ Scott sentence.  Then $I(G)$ is 
$\Sigma^0_3$.  If $G$ satisfies the hypothesis of Lemma \ref{sigma3hard}, then $I(G)$ is $m$-complete $\Sigma^0_3$, so the computable $\Sigma_3$ Scott sentence is best possible.  If we find a computable $d$-$\Sigma_2$ Scott sentence for $G$, then $I(G)$ is $d$-$\Sigma^0_2$.  If $G$ satisfies the very mild conditions of Lemma \ref{dsigma2hard}, then $I(G)$ is $m$-complete $d$-$\Sigma^0_2$, so the computable $d$-$\Sigma_2$ Scott sentence is best possible.  Bearing this in mind, we state some theorems.

\begin{thm}

Let $G$ be a computable subgroup of $\mathbb{Q}$, with $P^0$ and $P^\infty$ both finite, and $P^{fin}$ infinite.  Then $G$ has a computable $\Sigma_3$ Scott sentence, and $I(G)$ is $m$-complete $\Sigma^0_3$.

\end{thm}

\begin{proof}

We have a computable $\Sigma_3$ Scott sentence, and $I(G)$ is $\Sigma^0_3$.  Since $P^0$ and $P^\infty$ are both finite, $P^{fin}$ is co-finite, so we can apply Lemma \ref{sigma3hard}.  We get the fact that $I(G)$ is $m$-complete $\Sigma^0_3$.
\end{proof}

As a concrete example, consider $G \subseteq \mathbb{Q}$ such that $G \models (p_n^n \mid 1\ \wedge\ p_n^{n+1}\nmid 1)$, where $(p_n)_{n \in \omega}$ is the standard enumeration of the primes in order.  With this $G$, $P^{fin} = P$, and $P^0 = P^\infty = \emptyset$, so $I(G)$ is $m$-complete $\Sigma^0_3$, and the computable $\Sigma_3$ Scott sentence is best possible.

\begin{thm}

Suppose $P^0$ is infinite, $P^{fin}$ is infinite, and $P^\infty$ is finite. Then $G$ has a computable $\Sigma_3$ Scott sentence, and $I(G)$ is $m$-complete $\Sigma^0_3$.
\end{thm}

\begin{proof}
We have a computable $\Sigma_3$ Scott sentence, so $I(G)$ is $\Sigma^0_3$.  We know that $P^{fin} \cup P^\infty$ is $\Sigma^0_1$. Since $P^\infty$ is finite, this means $P^{fin}$ itself is $\Sigma^0_1$.  Since $P^{fin}$ is also infinite, we can conclude that $P^{fin}$ has a computable infinite subset. So, by Lemma \ref{sigma3hard}, $I(G)$ is $\Sigma^0_3$-hard, and thus $m$-complete $\Sigma^0_3$. 
\end{proof}

These theorems give us some examples of groups for which the computable $\Sigma_3$ Scott sentence is best possible.  There are other examples with computable $d$-$\Sigma_2$ Scott sentences.  Here is a condition on $G$ that guarantees this.

\begin{lem}
\label{dsigma2desc}

Let $G$ be a computable subgroup of $ \mathbb{Q}$. If $P^{fin}$ is finite, and $P^\infty$ is computable, then $G$ has a computable $d$-$\Sigma_2$ Scott sentence.  

\end{lem}

\begin{proof} 

Let $G$ be such a group. Then it is clear that $G$ is isomorphic to a computable group with the same $P^\infty$ set but with $P^{fin} = \emptyset$. Such an isomorphism is obtained by changing what element we consider to be $1$. Therefore, without loss of generality, we may assume that $P^{fin} = \emptyset$. Using this fact, we form a computable Scott sentence for $G$ by taking the conjunction of the computable $\Pi_2$ stating that $G$ is a torsion-free abelian group of rank $1$ with the following computable $d$-$\Sigma_2$ sentence:

\[ (\forall y) \displaystyle\bigdoublewedge_{p \in P^\infty, k \in \omega} (p^k \mid y) \wedge (\exists x) \displaystyle\bigdoublewedge_{p \notin P^\infty} (p \nmid x) \]

The $\Pi_2$ part of the sentence above says that every element is infinitely divisible by those elements in $P^\infty$. The $\Sigma_2$ part of the sentence says that there exists an element $x$ such that for each $p$ not in $P^\infty$, $p\nmid x$. This $x$ is a witness for $1$ in $G$.  
\end{proof}

We can use the results above to conclude that certain $G$ have a $d$-$\Sigma^0_2$ Scott sentence that is optimal.  

\begin{thm}

Let $G \subseteq \mathbb{Q}$ be a computable subgroup, with $P^0$ infinite, $P^{fin}$ finite, and $P^\infty$ finite. Then $G$ has a computable $d$-$\Sigma_2$ Scott sentence, and $I(G)$ is $m$-complete $d$-$\Sigma^0_2$.
\end{thm}

\begin{proof}

First, we show that $G$ has a computable $d$-$\Sigma_2$ Scott sentence.  Since $P^\infty$ is finite, it is computable, so by Lemma \ref{dsigma2desc}, we have a computable $d$-$\Sigma_2$ Scott sentence.  Then $I(G)$ is $d$-$\Sigma^0_2$.  To show hardness, first we assume that $P^\infty$ is nonempty, and we refer back to Lemma \ref{dsigma2hard}.  Since $P^0$ is infinite and thus nonempty, we can conclude that $I(G)$ is $d$-$\Sigma^0_2$-hard.  If $P^\infty = \emptyset$, then 
$G \cong \mathbb{Z}$, which is a finitely generated abelian group. By Theorem \ref{abcomp}, $I(G)$ is $m$-complete $d$-$\Sigma^0_2$.
\end{proof}

\begin{thm}

Suppose $P^0$ and $P^{fin}$ are finite, and $P^\infty$ is infinite but not all of $P$. Then $G$ has a computable $d$-$\Sigma_2$ Scott sentence, and $I(G)$ is $m$-complete $d$-$\Sigma^0_2$.
\end{thm}

\begin{proof}

If $P^\infty \ne P$, then $P^0 \cup P^{fin} \ne \emptyset$, so we can apply Lemma 2 to conclude that $I(G)$ is $d$-$\Sigma^0_2$-hard.  Since $P^0$ and $P^{fin}$ are finite, they are computable, so we can conclude that $P^\infty$ is also computable. Thus, by Lemma \ref{dsigma2desc}, $G$ has a computable $d$-$\Sigma_2$ Scott sentence, and $I(G)$ is $d$-$\Sigma_2$. So $I(G)$ is $m$-complete $d$-$\Sigma^0_2$.
\end{proof}

Note that if $P^\infty = P$, then $G \cong \mathbb{Q}$.  In \cite{idxsets}, it was shown that this group has a computable $\Pi_2$ Scott sentence, and $I(G)$ is $m$-complete $\Pi^0_2$.

\bigskip

Our results on subgroups of $\mathbb{Q}$ are summarized in the following table.

\begin{center}
\begin{tabular}{|c|c|c|c|}
\hline
Case & Description & Lower Bound & Upper Bound \\
\hline
1 & $P_0$ is infinite, $P_{fin}$ is finite, and $P_\infty$ is finite & $d$-$\Sigma^0_2$ & $d$-$\Sigma^0_2$ \\
\hline
2 & $P_0$ is finite, $P_{fin}$ is infinite, and $P_\infty$ is finite & $\Sigma^0_3$ & $\Sigma^0_3$ \\
\hline
3 & $P_0$ is finite, $P_{fin}$ is finite, and $P_\infty$ is infinite & $d$-$\Sigma^0_2$ & $d$-$\Sigma^0_2$ \\
\hline
4 & $P_0$ is finite, $P_{fin}$ is infinite, and $P_\infty$ is infinite & $d$-$\Sigma^0_2$ & $\Sigma^0_3$ \\
\hline
5 & $P_0$ is infinite, $P_{fin}$ is finite, and $P_\infty$ is infinite & $d$-$\Sigma^0_2$ & $\Sigma^0_3$ \\
\hline
6 & $P_0$ is infinite, $P_{fin}$ is infinite, and $P_\infty$ is finite & $\Sigma^0_3$ & $\Sigma^0_3$ \\
\hline
7 & $P_0$ is infinite, $P_{fin}$ is infinite, and $P_\infty$ is infinite & $d$-$\Sigma^0_2$ & $\Sigma^0_3$\\
\hline
\end{tabular}
\end{center}

\bigskip

We have completeness results for Cases (1), (2), (3), and (6), but the other cases remain open. There are some interesting open questions. We can show $\Sigma^0_3$ $m$-completeness of $I(G)$ if we know that $P^{fin}$ has an infinite computable subset, but what if this is not the case? That is, what if $P^{fin}$ is \emph{immune}? Since $P^{fin} \cup P^\infty$ is always $\Sigma^0_1$, we know it always has a computable infinite subset, but unfortunately we cannot generalize our proof technique, because we need to be able to distinguish between elements of $P^{fin}$ and elements of $P^\infty$.\footnote{There are recent results of Ho \cite{Ho} that give further information.}

Let $X \subseteq \N$ be computable from the halting set $\emptyset'$, or $X \leq_T \emptyset'$. Then recall that $X$ is \emph{low} if its jump $X'$ is Turing equivalent to $\emptyset '$, and \emph{high} if $X'$ is Turing equivalent to $\emptyset ''$. Similarly, $X$ is called \emph{$2$-high} if $X'' \equiv_T \emptyset ^{(3)}$. Then we have the following theorem.
 
\begin{prop}

Suppose $G$ is a computable subgroup of $\mathbb{Q}$, with $P^{fin} = \emptyset$ and $P^{\infty} = X$ (it follows that $X$ is c.e.).    
Then $G$ has a Scott sentence that is the conjunction of a computable $\Pi_2$ sentence and an ``$X$-computable'' $\Sigma_2$ sentence---there is a conjunction c.e.\ relative to $X$, but not c.e. Furthermore: 

\begin{enumerate}
\item  If $X$ is low, then $I(G)$ is $d$-$\Sigma^0_2$.
\item  If $X$ is not $2$-high, then $I(G)$ is not $m$-complete $\Sigma^0_3$.  
\end{enumerate}
\end{prop}

\begin{proof}

We have a Scott sentence that is the conjunction of the following:

\begin{enumerate}
\item  a computable $\Pi_2$ sentence describing the torsion-free abelian groups of rank~$1$,
\item  a computable $\Pi_2$ sentence saying that for all $p\in X$, all elements are divisible by $p$,
\item  an $X$-computable $\Sigma_2$ sentence saying that there exists $x$ such that $x$ is not divisible by any 
$p\notin X$. 
\end{enumerate}

Let $\varphi$ be the conjunction of the first two parts---this is computable $\Pi_2$.  The set of computable indices for groups satisfying $\varphi$ is $\Pi^0_2$.  Let $\psi$ be the third part---this is $X$-computable $\Sigma_2$.  It follows that the set of indices for computable groups satisfying $\psi$ is $\Sigma^0_2$ relative to $X$.  If $X$ is low, then this is $\Sigma^0_2$.  Now, $I(G)$ is the intersection of the $\Pi^0_2$ set of indices for groups satisfying $\varphi$ with the $\Sigma^0_2$ set of indices satisfying $\psi$, so it is $d$-$\Sigma^0_2$.  If $I(G)$ is $m$-complete $\Sigma^0_3$, then since $\emptyset^{(3)}$ is $\Sigma^0_3$, we would have 
$\emptyset^{(3)}\leq_m I(G)$, so $\emptyset^{(3)}\leq_T I(G)$.  We have seen that $I(G)$ is the intersection of a $\Pi^0_2$ set and a set that is $\Sigma^0_2$ relative to $X$.  It follows that $I(G)\leq_T X^{''}$.  If $X$ is not $2$-high, then $X^{''}$ is strictly below $\emptyset^{(3)}$ in Turing degree.      
\end{proof}

In the case where $X = P^\infty$ is low, we have shown that $I(G)$ is $d$-$\Sigma^0_2$, but we have been unable to find a computable $d$-$\Sigma_2$ Scott sentence.\footnote{In \cite{KMc}, it is shown that this group does not have a computable $d$-$\Sigma_2$ Scott sentence.}   
  
\section{Conclusion}

In the original version of this paper (written in 2013), we asked whether every finitely generated computable group $G$ has a computable $d$-$\Sigma_2$ Scott sentence. This was recently shown to be false by Harrison-Trainor and Ho \cite{HH}, who constructed an example of a finitely generated group for which a computable $\Sigma_3$ Scott sentence is optimal. However, their example was not finitely presented, so one may further ask whether every finitely \emph{presented} computable group has a computable $d$-$\Sigma_2$ Scott sentence. There is evidence for a positive answer in the earlier paper \cite{freegroups}, and in Section 2.  There is further evidence in more recent work by Raz \cite{raz} and Ho \cite{Ho}.          

\begin{conj}
\label{fingen}

Every finitely presented computable group $G$ has a computable $d$-$\Sigma_2$ Scott sentence.   

\end{conj}

In Section 3, we proved completeness results for the index sets of a number of different classes of subgroups of $\mathbb{Q}$. But there were also a number of different classes for which we merely provided upper and lower bounds.  In a version of this paper circulated earlier, we asked for more precise results for $G\subseteq\mathbb{Q}$ in the case where $P^0$ and $P^\infty$ both infinite, but $P^{fin} = \emptyset$.  We mentioned already a result in \cite{KMc} for the case where $P^\infty$ is a low non-computable c.e.\ set.  Ho \cite{Ho} gave results for the case where $P^\infty$ is high, in particular, when it is equal to the halting set.    


\newpage

%
%

\end{document}